\documentclass[12pt]{amsart}
\usepackage{amsmath,amsfonts,textcomp,pstricks}
\usepackage[dvips]{graphicx}

\newtheorem{thm}{Theorem }

\newtheorem{lemma}{Lemma }

\newtheorem{prop}{Proposition }

\newtheorem{corollary}{Corollary }

\newtheorem{con}{Conjecture}

\theoremstyle{definition}
\newtheorem{deff}{Definition }

\newtheorem{rem}{Remark }

\def\RR{{\mathbb R}}

\def\QQ{{\mathbb Q}}
\def\ZZ{{\mathbb Z}}

\def\FF{{\mathbb F}}
\def\fq{{\mathbb{F}_q}}
\def\kk{{\bar{k}}}

\def \bra#1\ket {\mathop{\vphantom{#1}\left<\smash{#1}\right>}\nolimits}

 \DeclareMathOperator{\Hom}{Hom}
\DeclareMathOperator{\End}{End}

 \DeclareMathOperator{\rk}{rk}

 \DeclareMathOperator{\Np}{Np}\DeclareMathOperator{\Hp}{Hp}
 
\DeclareMathOperator{\ord}{ord}

\def\plim{\mathop{{\lim\limits_{\longleftarrow}}}\nolimits}

\renewcommand \phi {\varphi}
\renewcommand \rho {\varrho}

\begin{document}
\author{Sergey Rybakov}
\thanks{Supported in part by RFBR grants no. 07-01-00051, 07-01-92211 and 08-07-92495}
\address{Poncelet laboratory (UMI 2615 of CNRS and Independent University of
Moscow)}
\address{Institute for information transmission problems of the Russian Academy of Sciences}
\email{rybakov@mccme.ru, rybakov.sergey@gmail.com}%

\title[The groups of points on abelian varieties]
{The groups of points on abelian varieties\\ over finite fields}
\date{}


\begin{abstract}
Let $A$ be an abelian variety with commutative endomorphism
algebra over a finite field $k$. The $k$-isogeny class of $A$ is
uniquely determined by a Weil polynomial $f_A$ without multiple
roots. We give a classification of the groups of $k$-rational
points on varieties from this class in terms of Newton polygons of
$f_A(1-t)$.
\end{abstract}

\maketitle

\section{Introduction}
Let $A$ be an abelian variety of dimension $g$ over a finite field
$k=\fq$, and let $A(k)$ be the group of $k$-rational points on
$A$. Tsfasman~\cite{Ts} classified all possible groups $A(k)$,
where $A$ is an elliptic curve (see the English exposition
in~\cite[3.3.15]{TsVN}). Later the same result was independently
proved in~\cite{Ru2}~and~\cite{Vol} using ~\cite{Sch}. Xing
obtained a similar classification when $A$ is a supersingular
simple surface~\cite{Xi1} and~\cite{Xi2}. In this paper such a
description is obtained for the groups of $k$-rational points on
abelian varieties with commutative endomorphism algebra.

For an abelian group $H$ we denote by $H_\ell$ the $\ell$-primary
component of $H$. Let $A(k)=\oplus_\ell A(k)_\ell$. We associate
to $A(k)_\ell$ a polygon of special type.

\begin{deff}Let $0\leq m_1\leq m_2\leq\dots\leq m_r$ be nonnegative
integers, and let $H=\oplus_{i=1}^r\ZZ/\ell^{m_i}\ZZ$ be an
abelian group of order $\ell^m$. The {\it Hodge polygon
$\Hp_\ell(H,r)$ of a group $H$} is the convex polygon with
vertices $(i,\sum_{j=1}^{r-i}m_j)$ for $0\leq i\leq r$. It has
$(0,m)$ and $(r,0)$ as its endpoints, and its slopes are
$-m_r,\dots, -m_1$.
\end{deff}

Note that Hodge polygon of a group depends on a choice of $r$, but
the zero slope of $\Hp_\ell(A(k)_\ell,r)$ is not interesting, and
we can use the most convenient choice of $r$. Mostly we need the
case $r=2g$, and write $\Hp_\ell(H)=\Hp_\ell(H,2g)$ (here we
assume that $H$ can be generated by $2g$ elements). For example,
let $g=1$, and $H=\ZZ/\ell\ZZ\oplus\ZZ/\ell\ZZ$, then
$\Hp_\ell(H,2)$ is a straight line (see Picture~$1$). Let
$H=\ZZ/\ell^2\ZZ$ be cyclic, then $\Hp_\ell(H,2)$ has a zero slope
(see Picture~$2$).

\begin{pspicture}(300pt,110pt)

\psline[linewidth=0.3pt](80pt,30pt)(10pt,30pt)(10pt,100pt)
\psline[linewidth=1.2pt](70pt,30pt)(10pt,90pt)

\pscircle*(10pt,30pt){1.5pt} \pscircle*(40pt,30pt){1.5pt}
\pscircle*(40pt,60pt){1.5pt} \pscircle*(10pt,90pt){1.5pt}
\pscircle*(70pt,30pt){1.5pt}

\rput(5pt,23pt){0} \rput(37pt,23pt){1} \rput(67pt,23pt){2}
\rput(3pt,87pt){2}

\rput(40pt,5pt){Pic. 1}

\psline[linewidth=0.3pt](180pt,30pt)(110pt,30pt)(110pt,100pt)
\psline[linewidth=1.2pt](170pt,30pt)(140pt,30pt)(110pt,90pt)

\pscircle*(110pt,30pt){1.5pt} \pscircle*(140pt,30pt){1.5pt}
\pscircle*(110pt,90pt){1.5pt} \pscircle*(170pt,30pt){1.5pt}

\rput(105pt,23pt){0} \rput(137pt,23pt){1} \rput(167pt,23pt){2}
\rput(103pt,87pt){2}

\rput(140pt,5pt){Pic. 2}

\psline[linewidth=0.3pt](280pt,30pt)(210pt,30pt)(210pt,100pt)
\psline[linewidth=1.2pt](270pt,30pt)(240pt,37pt)(210pt,90pt)
\psline[linewidth=1.2pt](270pt,30pt)(240pt,45pt)(210pt,90pt)

\pscircle*(210pt,30pt){1.5pt} \pscircle*(240pt,30pt){1.5pt}
\pscircle*(240pt,37pt){1.5pt} \pscircle*(240pt,45pt){1.5pt}
\pscircle*(210pt,90pt){1.5pt} \pscircle*(270pt,30pt){1.5pt}

\rput(205pt,23pt){0} \rput(237pt,23pt){1} \rput(267pt,23pt){2}
\rput(203pt,87pt){2}

\rput(240pt,5pt){Pic. 3} \rput(270pt,55pt){$\ord_\ell(b-2)$}

\end{pspicture}

Note also that isomorphism class of $H$ depends only on
$\Hp_\ell(H)$. For a polynomial $P\in\ZZ[t]$ we denote by
$\Np_\ell(P)$ the Newton polygon of $P$ with respect to $\ell$
(see Section $3$ for the precise definition). The aim of this
paper is to prove the following theorem.

\begin{thm}\label{rp}
Let $A$ be an abelian variety over a finite field with Weil
polynomial $f_A$ {\rm (}see Section $2$ for the definition of Weil
polynomial{\rm )}. Suppose $f_A$ has no multiple roots {\rm (}i.e.
endomorphism algebra $\End^{\circ}(A)$ is commutative{\rm)}. Let
$G$ be an abelian group of order $f_A(1)$. Then $G$ is a group of
points on some variety in the isogeny class of $A$ if and only if
$\Np_\ell(f_A(1-t))$ lies on or above $\Hp_\ell(G_\ell)$ for any
prime number $\ell$.
\end{thm}

As an example, we prove the following corollary of
Theorem~\ref{rp}. Originally it was proved in~\cite{Ts}. Later the
same result was independently proved in~\cite{Ru2}~and~\cite{Vol}
using ~\cite{Sch}.

\begin{corollary}\label{cor1}
Let $N=1-b+q$ be the order of $B(k)$ for an elliptic curve $B$.
Then $G=B(k)$ satisfies the following conditions.
\begin{enumerate}
    \item If $b\neq\pm2\sqrt{q}$, then $G\cong
    \ZZ/n_1\ZZ\oplus\ZZ/n_2\ZZ$, where $N=n_1n_2$, and $n_1$ divides
    $b-2$ and $n_2$.
    \item If $b=\pm2\sqrt{q}$, then $G\cong
    (\ZZ/n_1\ZZ)^2$, and $N=n_1^2$.
\end{enumerate}
If a finite commutative group $G$ satisfies $(1)$ or $(2)$, then
there exists an elliptic curve $B'$ isogenous to $B$ such that
$B'(k)\cong G$.
\end{corollary}
\begin{proof}
If $b\neq\pm2\sqrt{q}$, then $f_B$ has no multiple roots. Fix a
prime $\ell$. Let
$B(k)_\ell\cong\ZZ/\ell^{m_1}\ZZ\oplus\ZZ/\ell^{m_2}\ZZ$, and let
$m_1\leq m_2$. Since $f_B(1-t)=t^2+(b-2)t+(1-b+q)$, we have by
Theorem~\ref{rp} that $m_1\leq\ord_\ell(b-2)$ (see Picture $3$).
Equivalently, $\ell^{m_1}$ divides $b-2$, and the first case
follows. The second case is obvious since $F$ acts on $T_\ell$ as
multiplication by $b/2=\pm\sqrt{q}$. Conversely, if $G$ satisfies
$(1)$, then the existence of an elliptic curve $B'$ with
$B'(k)\cong G$ follows from Theorem~\ref{rp} combined with the
inequality $m_1 \leq \ord_\ell(b-2)$.
\end{proof}

\begin{rem}
Originally the Tsfasman theorem relies on the Waterhouse
classification of isogeny classes of elliptic curves~\cite{Wa} and
has $6$ cases. We make the statement shorter, but it becomes a
little different. For example, for a supersingular curve $B$ with
commutative endomorphism algebra it is not immediately clear that
$B(k)$ is cyclic modulo $2$-torsion.
\end{rem}

Corollary~\ref{cor1} allows us to construct an example of an
isogeny class with the following properties:
\begin{enumerate}\item The Weil polynomial $f$ has multiple roots;
\item not any abelian group $G$ of order $f(1)$ such that
$\Np_\ell(f(1-t))$ lies on or above $\Hp_\ell(G_\ell)$ for any
prime number $\ell$ is a group of points on some variety from the
isogeny class.
\end{enumerate}
Indeed, take $f(t)=(t-3)^2$. Then by Corollary~\ref{cor1} for any
$A$ from the isogeny class $A(\FF_9)\cong(\ZZ/2\ZZ)^2$. It follows
that the isogeny class does not contain an elliptic curve, whose
group of points is cyclic.

The author is grateful to M.A.~Tsfasman and Yu.G.~Zarhin for their
attention to this work and to reviewer for providing useful
corrections and comments on the paper.

\section{Preliminaries}
Throughout this paper $k$ is a finite field $\fq$ of
characteristic $p$. Let $A$ and $B$ be abelian varieties over $k$.
It is well known that the group $\Hom(A,B)$ of $k$-homomorphisms
from $A$ to $B$ is finitely generated and torsionfree. We use the
following notation: $\Hom^{\circ}(A,B)=\Hom(A,B)\otimes_{\ZZ}\QQ$,
and $\End^{\circ}(A)=\End(A)\otimes_{\ZZ}\QQ$. The algebra
$\End^{\circ}(A)$ contains the Frobenius endomorphism $F$,
 and its center is equal to $\QQ[F]$. Thus
$\End^{\circ}(A)$ is commutative if and only if
$\End^{\circ}(A)=\QQ[F]$.

Let $A$ be an abelian variety of dimension $g$ over $k$, and let
$\kk$ be an algebraic closure of $k$. For a natural number $m$
denote by $A_m$ the kernel of multiplication by $m$ in $A(\kk)$.
Let $A[m]$ be the group subscheme of $A$, which is the kernel of
multiplication by $m$. By definition $A_m=A[m](\kk)$. Let
$T_\ell(A) = \plim A_{\ell^r}$ be the Tate module, and
$V_\ell(A)=T_\ell(A)\otimes_{\ZZ_\ell}\QQ_\ell$ be the
corresponding vector space over $\QQ_\ell$. If $\ell\ne p$, then
$T_\ell(A)$ is a free $\ZZ_\ell$-module of rank $2g$. The
Frobenius endomorphism $F$ of $A$ acts on the Tate module by a
semisimple linear operator, which we also denote by $F:
T_\ell(A)\to T_\ell(A)$. The characteristic polynomial
$$
f_A(t) = \det(t-F|T_\ell(A))
$$
is called a {\it Weil polynomial of $A$}. It is a monic polynomial
of degree $2g$ with rational integer coefficients independent of
the choice of prime $\ell\ne p$. It is well known that for
isogenous varieties $A$ and $B$ we have $f_A(t)=f_B(t)$. Moreover,
Tate proved that the isogeny class of abelian variety is
determined by its characteristic polynomial, that is
$f_A(t)=f_B(t)$ implies that $A$ is isogenous to $B$. The
polynomial $f_A$ has no multiple roots if and only if the
endomorphism algebra $\End^{\circ}(A)$ is commutative (see
\cite{WM}).

If $\ell=p$, then $T_p(A)$ is called a {\it physical Tate module}.
In this case, $f_A(t)=f_1(t)f_2(t)$, where $f_1,f_2\in\ZZ_p[t]$,
and $f_1(t)=\det(t-F|T_p(A))$. Moreover $d=\deg f_1\leq g$, and
$f_2(t)\equiv t^{2g-d}\mod p$ (see~\cite{De}).

We say that $\phi:B\to A$ is an {\it $\ell$-isogeny}, if degree of
$\phi$ is a power of $\ell$. The following lemma is well known,
but we prove it here for the sake of completeness.

\begin{lemma}
\label{lem_on_Tate_module} If $\phi:B\to A$ is an isogeny then
$T_\ell(\phi):T_\ell(B)\to T_\ell(A)$ is a $\ZZ_\ell$-linear
embedding commuting with the action of the Frobenius endomorphisms
and if $T$ denotes its image then
\begin{equation}\label{t}
F(T)\subset T\quad\text{and}\quad T\otimes_{\ZZ_\ell}\QQ_\ell
\cong T_\ell(A)\otimes_{\ZZ_\ell}\QQ_\ell.
\end{equation}
Conversely, if $T\subset T_\ell(A)$ is a $\ZZ_\ell$-submodule such
that $(\ref{t})$ holds, then there exists an abelian variety $B$
over $k$, and an $\ell$-isogeny $\phi:B\to A$ such that
$T_\ell(\phi)$ induces an isomorphism $T_\ell(B)\cong T$.
\end{lemma}
\begin{proof}
The first part is obvious. The second part can be proved as
follows. First, (\ref{t}) implies that there exists $k\in\ZZ$ such
that $\ell^kT_\ell(A)\subset T$. Further, note that
$T_\ell(A)/\ell^kT_\ell(A)\cong A_{\ell^k}$. Hence the group
$T/\ell^kT_\ell(A)$ can be considered as a subgroup in
$A_{\ell^k}\subset A(\kk)$. Moreover, since $F(T)\subset T$ it
follows that $T/\ell^kT_\ell(A)$ is invariant under the action of
the Frobenius, and thus defines a certain group subscheme $G$ of
$A$. If $\ell\neq p$, we define $B=A/G$. If $\ell=p$, there is the
canonical decomposition $A[p^k]=G_r\oplus G_l$, where $G_r$ is
reduced and $G_r(\kk)=A_{p^k}$, and $G_l(\kk)=0$~\cite{De}. In
this case, we define $B=A/(G\oplus G_l)$. It is clear that $B$ is
defined over $k$, and $A\cong B/G'$, where $G'$ is reduced and
$G'(\kk)=T_\ell(A)/T$. This gives a desired isogeny $\phi:B\to A$.
\end{proof}

\section{The groups of points}
\label{np} Let $Q(t)=\sum_i Q_i t^i$ be a polynomial of degree $d$
over $\QQ_\ell$, and let $Q(0)=Q_0\neq 0$. Take the lower convex
hull of the points $(i,\ord_\ell(Q_i))$ for $0\leq i\leq d$ in
$\RR^2$. The boundary of this region without vertical lines is
called {\it the Newton polygon $\Np_\ell(Q)$ of $Q$}. Its vertices
have integer coefficients, and $(0,\ord_\ell(Q_0))$ and
$(d,\ord_\ell(Q_d))$ are its endpoints.

Let $E$ be an injective endomorphism of $T_\ell(A)$, and let
$H=T_\ell(A)/E T_\ell(A)$ be its cokernel, which is a finite
$\ell$-group. Define the {\it Hodge polygon of the endomorphism
$E$} as the Hodge polygon $\Hp_\ell(H,\rk T_\ell(A))$ of the group
$H$. We need the following simple result.

\begin{prop}\label{prop1}
The Hodge polygon $\Hp_\ell(A(k)_\ell,\rk T_\ell(A))$ is equal to
the Hodge polygon of $1-F$.
\end{prop}
\begin{proof}
Since $A(k)_\ell$ is finite, there exists a positive integer $N$
such that $A(k)_\ell\subset A_{\ell^N}$. Clearly,
$A(k)_\ell=\ker(1-F:A_{\ell^N}\to A_{\ell^N})$. Apply $1-F$ to the short
exact sequence:
$$0\to T_\ell(A)\xrightarrow{\ell^N}T_\ell(A)\to A_{\ell^N}\to 0.$$
By snake lemma we get:
$$0=\ker(1-F:T_\ell(A)\to T_\ell(A))\to A(k)_\ell\to T_\ell(A)/(1-F)T_\ell(A)
\xrightarrow{0}T_\ell(A)/(1-F)T_\ell(A).$$ Thus $A(k)_\ell\cong
T_\ell(A)/(1-F)T_\ell(A)$, and the proposition follows.
\end{proof}

We see that the group $A(k)_\ell$ depends only on the action of
$F$ on the Tate module $T_\ell(A)$. In particular, order of $A(k)$
is $f_A(1)$. Thus we reduced our task to the following linear
algebra problem. We are given a $\QQ_\ell$-vector space $V$ of
finite positive dimension $d$ and an invertible linear operator
$E: V \to V$, whose characteristic polynomial $f(t)=\det (E-t|V)$
lies in $\ZZ_\ell[t]$. We want to describe all finite commutative
groups (up to isomorphism) of the form $T/E T$, where $T$ is an
arbitrary $E$-invariant $\ZZ_\ell$-lattice of rank $d$ in $V$. We
solve this problem under the additional assumption that $f$ has no
multiple roots.

The next statement establishes a connection between the Hodge
polygon of an endomorphism and the Newton polygon of its
characteristic polynomial.

\begin{thm}\cite[4.3.8]{Ke2008}\cite[8.40]{BO}
Let $E$ be an injective endomorphism of a free $\ZZ_\ell$-module
of finite rank. Let $f(t)=\det(E-t|V)$ be its characteristic
polynomial. Then $\Np_\ell(f)$ lies on or above the Hodge polygon
of $E$, and these polygons have same endpoints.
\end{thm}

Recall that if $\ell=p$, then $f_A(t)=f_1(t)f_2(t)$, and
$f_2(t)\equiv t^{2g-d}\mod p$. Thus the only slope of
$\Np_p(f_2(1-t))$ is zero, and $\Np_p(f_A(1-t))$ equals to
$\Np_p(f_1(1-t))$ up to this zero slope.

\begin{corollary}{\label{rp1}}
Let $A$ be an abelian variety $A$ with Weil polynomial $f_A$. Then
$\Np_\ell(f_A(1-t))$ lies on or above $\Hp_\ell(A(k)_\ell)$, and
these polygons have same endpoints $(0,\ord_\ell(f_A(1))$ and
$(2g,0)$.
\end{corollary}

Now we are ready to solve our linear algebra problem.

\begin{thm}{\label{rp2}}
Let $V$ be a $\QQ_\ell$-vector space of positive finite dimension
$d$, and let $E: V \to V$ be an invertible linear operator such
that the characteristic polynomial $f(t)=\det (E-t|V)$ lies in
$\ZZ_\ell[t]$. Suppose $f$ has no multiple roots. Let $G$ be an
abelian group of order $\ell^m$, where $m=\ord_\ell(f(0))$. If
$\Np_\ell(f)$ lies on or above $\Hp_\ell(G,d)$, then there exists
an $E$-invariant $\ZZ_\ell$-lattice $T$ of rank $d$ in $V$ such
that $T/ET\cong G$.
\end{thm}

\begin{corollary}{\label{rpc}}
We keep the notation of theorem~\ref{rp2}. Suppose $f$ has no
multiple roots. Then the group $G$ is isomorphic to $T/E T$ for
some $E$-invariant $\ZZ_\ell$-lattice $T$ of rank $d$ in $V$ if
and only if $\Np_\ell(f)$ lies on or above $\Hp_\ell(G,d)$, and
these polygons have same endpoints $(0,\ord_\ell(f(0))$ and
$(d,0)$.
\end{corollary}

\begin{proof}[Proof of theorem~\ref{rp2}]
Since $f$ has no multiple roots, there is an isomorphism of
$\QQ_\ell$-vector spaces $V\cong\QQ_\ell[t]/f(t)\QQ_\ell[t]$ such
that $E$ becomes multiplication by $t$ in
$\QQ_\ell[t]/f(t)\QQ_\ell[t]$. Let $x$ be an image of $t$ in
$\QQ_\ell[t]/f(t)\QQ_\ell[t]$, and let $R=\ZZ_\ell[x]$ be the
$\ZZ_\ell$-subalgebra of $\QQ_\ell[t]/f(t)\QQ_\ell[t]$ generated
by $x$. Then $R$ is a $\ZZ_\ell$-lattice in
$\QQ_\ell[t]/f(t)\QQ_\ell[t]$; in particular, the natural map
$R\otimes_{\ZZ_\ell}\QQ_\ell\to\QQ_\ell[t]/f(t)\QQ_\ell[t]$ is an
isomorphism of $\QQ_\ell$-vector spaces. We have to find an
$R$-submodule $T\subset R\otimes_{\ZZ_\ell}\QQ_\ell$ of
$\ZZ_\ell$-rank $d$ such that $x$ acts on $T$ with Hodge polygon
$\Hp_\ell(G,d)$.

Suppose $G=\oplus_{i=1}^d\ZZ/\ell^{m_i}\ZZ$, where $m_1\leq
m_2\leq\dots\leq m_d$ (recall that $m_i$ may be zero for some
$i$). First we construct certain elements $v_0,\dots, v_d\in
R\otimes_{\ZZ_\ell}\QQ_\ell$. Let us put $M(0)=0$, and
$M{(s)}=\sum_{i=1}^s m_i$ for $s\geq 1$. Let $f(t)=\sum_{i=0}^d
a_it^i$. We let
$$v_s=\frac{a_dx^s+\sum_{j=1}^s a_{d-j}x^{s-j}}{\ell^{M{(s)}}},$$
in particular, $v_d=f(x)/\ell^m=0$. Second, we define $T$ as the
$\ZZ_\ell$-submodule of $R\otimes_{\ZZ_\ell}\QQ_\ell$ generated by
$d$ elements $v_0, \dots , v_{d-1}$. Note that $v_0,v_2,\dots,
v_{d-1}$ have different degrees viewed as polynomials in $x$ (and
all degrees are strictly less than $d$), and hence form a basis of
$T$ over $\ZZ_\ell$. In particular, $T$ is a free
$\ZZ_\ell$-module of rank $d$. Notice that $v_0=a_d=(-1)^d.$ In
particular, $T$ contains $\ZZ_\ell\cdot 1$.

Now we prove that $T$ is an $R$-submodule of
$R\otimes_{\ZZ_\ell}\QQ_\ell$. The point $(d-s,{M{(s)}})$ is a
vertex of $\Hp_\ell(G,d)$. By assumption $\Np_\ell(f)$ lies on or
above $\Hp_\ell(G,d)$; thus $(d-s,{M{(s)}})$ is not higher than
$\Np_\ell(f)$. It follows that $\ell^{M{(s)}}$ divides $a_{d-s}$,
and
$$u_s=\frac{a_{d-s}}{\ell^{M{(s)}}}\cdot 1\in \ZZ_\ell\cdot 1\subset T,$$
thus for $s\geq 1$
\begin{equation}\label{eq2}xv_{s-1}=\ell^{m_s}(v_s-u_s)\in \ell^{m_s}T.
\end{equation}
This proves that $xT\subset T$.

Recall that, $v_0=a_d=(-1)^d$; clearly $m=\ord_\ell(\det
E)=\ord_\ell(a_0)$; thus $u_d=a_0/\ell^m$ is invertible in
$\ZZ_\ell$, and $$v_1-u_1,\dots, v_d-u_d=-u_d$$ is a basis of the
free $\ZZ_\ell$-module $T$. This gives us the natural surjective
map $\ZZ^d\to T/xT$. It follows from (\ref{eq2}) that this map
factors through a surjective map $G \to T/xT$. We conclude that
$T/x T\cong G$, since orders of both groups are equal to $\ell^m$.
\end{proof}

\begin{proof}[Proof of theorem~\ref{rp}.]
The ``if'' part follows from Corollary~\ref{rp1}. Let us prove the
``only if'' part.

For a given prime number $\ell$ and an abelian variety $A'$
isogenius to $A$ we let $V=V_\ell(A')$ and $E=1-F$. If $\ell\ne
p$, we let $f(t)=f_A(1-t)$, and if $\ell=p$, we let
$f(t)=f_1(1-t)$. By Theorem~\ref{rp2} there exists an $E$-invarint
lattice $T$ in $V$ such that $T/ET\cong G_\ell$. Clearly, $T$ is
$F$-invariant, and by Lemma~\ref{lem_on_Tate_module} there exists
an abelian variety $B'$ and an $\ell$-isogeny $B'\to A'$ such that
$T_\ell(B')\cong T$. By Proposition~\ref{prop1} we have
$G_\ell\cong B'(k)_\ell$.

Let $\ell_1,\dots,\ell_s$ be the set of prime divisors of
$f_A(1)$. It follows that there exists a sequence of isogenies
$$B=B_s\xrightarrow{\phi_s} B_{s-1}\to\dots\xrightarrow{\phi_2}
B_1\xrightarrow{\phi_1} A$$ such that $\phi_i:B_i\to B_{i-1}$ is
an $\ell_i$-isogeny and
$$B_i(k)_{\ell_i}\cong G_{\ell_i}.$$ Since $\phi_i$ is an
$\ell_i$-isogeny, $T_\ell(B_i)\cong T_\ell(B_{i-1})$ for any
$\ell\neq\ell_i$. Thus $B(k)\cong G$.
\end{proof}

\section{Noncommutative endomorphism algebras}

If $\End A$ is not commutative then we may apply the following
construction. Let $f_A=\prod_{j=1}^{s} f_j$, where all $f_i$ are
polynomials with integer coefficients, and $f_j$ divides
$f_{j-1}$. Suppose $f_j$ has no multiple roots for $1\leq j\leq
s$. Let $G_j$ be a family of abelian $\ell$-groups for $1\leq
j\leq s$ such that $\Np_\ell(f_j(1-t))$ lies on or above
$\Hp_\ell(G_j,\deg f_j)$. By Theorem~\ref{rp2} and
Lemma~\ref{lem_on_Tate_module}, we can construct modules $T_j$ and
an abelian variety $B$ with Tate module $T_\ell(B)=\oplus T_j$
such that $B(k)_\ell\cong\oplus G_j$. We have the following
conjecture.

\begin{con}
Let $f_A=\prod_{j=1}^{s} f_j$, where $f_j$ divides $f_{j-1}$, and
suppose that $f_j$ has no multiple roots for $1\leq j\leq s$. If
$\deg f_j\leq 2$ for all $j$, then $A(k)_\ell\cong \oplus G_j$,
where $G_j$ are $\ell$-primary abelian groups such that
$\Np_\ell(f_j(1-t))$ lies on or above $\Hp_\ell(G_j,\deg f_j)$ for
all $1\leq j\leq s$.
\end{con}

This conjecture is proved in~\cite{Xi1} for simple abelian
surfaces. However, there is an example of the group of points on
an abelian variety $A$ with $\deg f_1=3$, and $\deg f_2=1$ such
that this group is not a direct sum of two groups $G_1$ and $G_2$
such that $\Np_\ell(f_j(1-t))$ lies on or above $\Hp_\ell(G_j,\deg
f_j)$ for $j=1,2$. Let us consider the Weil polynomial
$f(t)=(t^2-2t+9)(t+3)^2$, and let $A$ be an abelian surface such
that $f_A=f$. Then $f(1-t)=(t^2+8)(t-4)^2$. Let $v_1,v_2,v_3,v_4$
be a basis of $V_2(A)$ such that $(1-F)v_1=2v_2$,
$(1-F)v_2=-4v_1$, $(1-F)v_3=4v_3$, and $(1-F)v_4=4v_4$. Let $T$ be
the $\ZZ_2$-submodule of $V_2(A)$ generated by $u_1=v_1+v_3,
u_2=-4v_2+4v_3,u_3=v_2+v_4$, and $u_4=4v_1+2v_4$. By
Lemma~\ref{lem_on_Tate_module} there exists an abelian surface $B$
such that $T_2(B)\cong T$. We claim that
$$B(\FF_9)_2\cong\ZZ/8\ZZ\oplus\ZZ/16\ZZ.$$ Indeed,
\begin{gather*}
(1-F)u_1=4u_1+2u_3-u_4\in T\\
(1-F)u_2=16u_3\in T\\
(1-F)u_3=4u_1-u_2+4u_3\in T\\
(1-F)u_4=8u_1\in T.\\
\end{gather*}

Theorem~\ref{rp} and Conjecture~$1$ allows one to classify the
groups of points on simple abelian surfaces. However the author
does not know even a conjectural classification of groups of
points on nonsimple surfaces.


\begin{thebibliography}{99}
\bibitem[BO]{BO} Berthelot P., Ogus A., Notes on crystalline
cohomology, Princeton University Press, Princeton, N.J.;
University of Tokyo Press, Tokyo, 1978.

\bibitem[De78]{De} Demazure M.,
Lectures on $p$-divisible groups,
 Lecture notes in mathematics 302, Springer-Verlag, Berlin, 1972.

\bibitem[Ke09]{Ke2008}Kedlaya K., $p$-adic differential
equations, Cambridge University Press, to appear, web draft
(30.09.2009) available at
http://www-math.mit.edu/\verb"~"kedlaya/papers/pde.pdf

\bibitem[Ru87]{Ru2} R\"uck H.-G.,
 A note on elliptic curves over finite fields.
    Math. Comp.  1987, 49  no. 179, 301--304.

\bibitem[Sch87]{Sch} Schoof R.,
 Nonsingular plane cubic curves over finite fields,
 J. Combin. Theory Ser. A, 1987,  46  no. 2, 183--211.

\bibitem[Ts85]{Ts} Tsfasman M.~A.,
The group of points of an elliptic curve over a finite field,
Theory of numbers and its applications, Tbilisi, 1985, 286-287.

\bibitem[TsVN07]{TsVN}
Tsfasman M., Vladut S., Nogin D., Algebraic geometric codes: basic
notions, Mathematical Surveys and Monographs, 139. American
Mathematical Society, Providence, RI, 2007.

\bibitem[Vo88]{Vol} Voloch J.~F., A note on elliptic curves over finite
fields, Bull. Soc. Math. France, 1988, 116  no. 4, 455--458.

\bibitem[Wa69]{Wa} Waterhouse W.,
Abelian varieties over finite fields, Ann.\ scient.\ \'Ec.\ Norm.\
Sup., 1969, 4 serie 2, 521--560.

\bibitem[WM69]{WM} Waterhouse W., Milne J.,
Abelian varieties over finite fields, Proc. Sympos. Pure Math.,
Vol. XX, State Univ. New York, Stony Brook, N.Y., 1969, 53--64.

\bibitem[Xi94]{Xi1} Xing Ch.,
The structure of the rational point groups of simple abelian
varieties of dimension two over finite fields, Arch.\ Math.,
1994,63, 427--430.

\bibitem[Xi96]{Xi2} Xing Ch.,
On supersingular abelian varieties of dimension two over finite
fields, Finite Fields Appl., 1996, 2, no. 4, 407--421.
\end{thebibliography}
\end{document}